\documentclass[11pt]{amsart}

\usepackage[margin=1.1in]{geometry}

\usepackage{amsmath,amssymb,enumitem,url}
\usepackage{amsthm}
\usepackage{thmtools,thm-restate}
\usepackage{graphicx}

\newlist{casesp}{enumerate}{2}
\setlist[casesp]{align=left, listparindent=\parindent, parsep =
  \parskip, font=\normalfont\bfseries, leftmargin=0pt,
  labelwidth=0pt, itemindent=.4em, labelsep=.4em, partopsep=0pt}
\setlist[casesp,1]{label=Case~\arabic*:,ref=\arabic*}
\setlist[casesp,2]{label=Case~\thecasespi\alph*:,ref=\thecasespi\alph*}

\declaretheorem{theorem}

\declaretheorem[sibling=theorem]{question}
\declaretheorem[sibling=theorem]{corollary}

\declaretheorem[sibling=theorem, style=definition]{definition}

\declaretheorem[style=remark, numbered=no]{remark}

\subjclass[2010]{Primary 37F15; Secondary 30C20, 30F45}

\newcommand{\C}{\mathbb{C}}
\newcommand{\CC}{\hat{\C}}
\newcommand{\D}{\mathbb{D}}
\newcommand{\N}{\mathbb{N}}

\newcommand{\cO}{\mathcal{O}}
\newcommand{\cS}{\mathcal{S}}

\DeclareMathOperator{\diam}{diam}
\DeclareMathOperator{\dist}{dist}

\DeclareMathOperator{\Crit}{Crit}

\begin{document}
\title[Expanding metrics for unicritical semihyperbolic
  polynomials]{Expanding metrics for unicritical \\[1ex]
    semihyperbolic polynomials}
\author{Lukas Geyer}
\address{Montana State University \\ Department of Mathematical
  Sciences \\ Bozeman, MT 59717--2400, USA \\ ORCiD ID:
  \url{https://orcid.org/0000-0001-5889-9037}}
\email{geyer@montana.edu}

\begin{abstract}
  We prove that unicritical polynomials $f(z)=z^d+c$ which are
  semihyperbolic, i.e., for which the critical point $0$ is a
  non-recurrent point in the Julia set, are uniformly expanding on the
  Julia set with respect to the metric $\rho(z) |dz|$, where
  $\rho(z) = 1+\dist(z,P(f))^{-1+1/d}$ has singularities on the
  postcritical set $P(f)$. We also show that this metric is
  H\"older equivalent to the usual Euclidean metric.
\end{abstract}

\maketitle

\tableofcontents

\section{Introduction}
In the theory of dynamics of polynomials and rational functions of
degree at least 2 in one complex dimension, by far the best understood
class is that of hyperbolic maps. Various weaker notions of
hyperbolicity have also been introduced and investigated. In order to
state definitions, results, and questions, let us introduce some
fairly standard notation. For background in complex dynamics, see one
of the many excellent textbooks, e.g., \cite{CarlesonGamelin1993} or
\cite{Milnor2006}.

\subsection{Notation and background}
For a rational map $f: \CC \to \CC$ of degree $d \ge 2$ we denote its
forward iterates by $f^n$, its set of critical points by $\Crit(f)$,
its \emph{postcritical set} by $P(f) = \overline{\bigcup_{n=1}^\infty
  f^n(\Crit f)}$, and its Julia set by $J(f)$. A \emph{periodic cycle}
of period $q \ge 1$ is a $q$-tuple of distinct points $Z = (z_0, z_1,
\ldots, z_{q-1})$ with $f(z_k) = z_{k+1}$ and $f(z_{q-1}) = z_0$. Its
\emph{multiplier} is $\lambda = (f^q)'(z_0) = (f^q)'(z_k)$. Since the
multiplier is invariant under M\"obius conjugacy, it is also defined
for cycles containing $\infty$. A periodic cycle $Z$ with multiplier
$\lambda$ is \emph{attracting} if $|\lambda|<1$, \emph{repelling} if
$|\lambda|>1$, and \emph{indifferent} if $|\lambda|=1$. An indifferent
cycle is \emph{rationally indifferent} (or \emph{parabolic}) if
$\lambda$ is a root of unity, \emph{irrationally indifferent}
otherwise.  A point $z$ is \emph{attracted} to the cycle $Z$ if
$\dist(f^n(z),Z) \to 0$ as $n \to \infty$. The \emph{$\omega$-limit
  set} $\omega(z) = \omega_f(z)$ of a point $z$ under $f$ is the set
of all limits $\lim_{k\to\infty} f^{n_k}(z)$ for sequences $n_k \to
\infty$.  A point $z \in \CC$ is \emph{non-recurrent} if it has a
neighborhood which does not contain any of its forward iterates
$f^n(z)$ for $n \ge 1$, or equivalently if $z \notin \omega(z)$. For
polynomials we denote the \emph{basin of infinity} by $A(\infty) =
A_f(\infty)$.
\begin{definition}
  Let $f:\CC \to \CC$ be a rational function of degree $d \ge 2$. Then
  $f$ is
  \begin{itemize}
  \item \emph{hyperbolic} if each critical points of $f$ is attracted
    by an attracting periodic cycle;
  \item \emph{subhyperbolic} if each critical point in the Fatou set
    of $f$ is attracted by an attracting periodic cycle, and each
    critical point in the Julia set of $f$ is strictly preperiodic;
  \item \emph{semihyperbolic} if each critical point in the Fatou set
    of $f$ is attracted by an attracting periodic cycle, and each
    critical point in the Julia set of $f$ is non-recurrent.
  \end{itemize}
\end{definition}
It is obvious that hyperbolic $\implies$ subhyperbolic $\implies$
semihyperbolic.  We are particularly interested in the question how
these notions are related to expansion on the Julia set $J(f)$, in the
following sense:
\begin{definition}
  \label{def:expansion}
  Let $f:\CC \to \CC$ be a rational function of degree $d \ge 2$ with
  $\infty \notin J(f)$, let $U$ be an open connected neighborhood of
  $J(f)$, and let $\rho:U \to (0,+\infty]$ be a continuous
  function. We denote by $d_\rho$ the induced path metric on $U$,
  defined by
  \[
    d_\rho(z_0,z_1) = \inf \left\{ \int_\gamma \rho(z)|dz| : \gamma
      \subseteq U \text{ piecewise smooth curve from } z_0 \text{ to }
      z_1 \right\}.
  \]
  Then
  \begin{enumerate}
  \item \label{it:expanding}
    $f$ is \emph{expanding} with respect to $\rho(z)|dz|$ if there
    exist $C>0$ and $\lambda > 1$ such that
    \begin{equation*}
      \frac{|(f^n)'(z)| \rho(f^n(z))}{\rho(z)} \ge C \lambda^n
    \end{equation*}
    for a dense set of $z \in J(f)$ and all $n \ge 1$.
  \item $\rho(z)|dz|$ \emph{induces the standard topology} (or is
    \emph{topologically equivalent} to the Euclidean metric) on $J(f)$
    if the identity map from $(J(f),d_\rho)$ to $J(f)$ equipped with
    the standard Euclidean metric is a homeomorphism.
  \item $\rho(z)|dz|$ is \emph{H\"older equivalent} to the Euclidean
    metric on $J(f)$ if the identity map from $(J(f),d_\rho)$ to
    $J(f)$ equipped with the standard Euclidean metric is bi-H\"older,
    i.e., it is H\"older continuous with H\"older continuous inverse.
  \end{enumerate}
\end{definition}
\begin{remark}
  The assumption that $\infty \notin J(f)$ is only added for
  convenience. In the case where $\infty \in J(f)$, one should replace
  the Euclidean metric with the spherical metric. Note that
  $d_\rho(z_0, z_1)$ might a priori be infinite.
\end{remark}
The following results are known, see e.g.\
\cite[Ch.~V]{CarlesonGamelin1993} or \cite[\S~19]{Milnor2006}.
\begin{theorem}
  \label{thm:subhyp-metric}
  A rational map $f$ with $\infty \notin J(f)$ is hyperbolic iff it is
  expanding with respect to the standard Euclidean metric. It is
  subhyperbolic iff it is expanding with respect to a metric
  $\rho(z)|dz|$ whose metric density $\rho(z)$ is a continuous
  function in an open connected neighborhood $U$ of $J(f)$ except for
  finitely many singularities $a_1, \ldots, a_q$, such that there
  exists $C>0$ and $\beta < 1$ with
  \begin{equation}
    \label{eq:subhyp-metric}
    C^{-1} \le \rho(z) \le \sum_{k=1} \frac{C}{|z-a_k|^\beta}
  \end{equation}
  for all $z \in U \setminus \{ a_1, \ldots, a_q \}$.
\end{theorem}
\begin{remark}
  In the case where $\infty \in J(f)$, a similar result with the
  spherical metric instead of the Euclidean metric is still true. In
  the subhyperbolic case, any metric $\rho(z)|dz|$ where $\rho$
  satisfies \eqref{eq:subhyp-metric} is H\"older equivalent to the
  Euclidean metric, so in particular it induces the same topology.
\end{remark}

\subsection{Results and questions on expanding metrics}
If one allows infinitely many singularities of the metric
$\rho(z)|dz|$, the following is a natural question.
\begin{question}
  Which rational maps $f$ admit expanding metrics $\rho(z)|dz|$
  topologically equivalent (or H\"older equivalent) to the spherical
  metric (or Euclidean metric in the case $\infty \notin J(f)$)?
\end{question}
\begin{remark}
  It is probably easier to first ask this only for polynomials. As far
  as the author knows, the answer is not known even in the quadratic
  family $f_c(z) = z^2+c$.
\end{remark}

In \cite[p.~9, Remark 2]{CarlesonJonesYoccoz1994}, Carleson, Jones,
and Yoccoz claim that an expanding metric can easily be constructed
for semihyperbolic polynomials. However, they only give a formula for
the case of quadratic polynomials, and they do not provide a proof
either of expansion or of the fact that this metric is topologically
equivalent to the Euclidean metric. Later, in his Ph.~D.\ thesis
\cite{Carette1997}, Carette provided a proof of expansion with respect
to the metric $\dist(z,P(f))^{-1/2}|dz|$ in the case of quadratic
polynomials, combining expansion on the postcritical set given by a
result of M\~a\'ne from \cite{Mane1993} with expansion in the
complement of the postcritical set with respect to its hyperbolic
metric. However, Carette's thesis was never published, and the part of
his proof that the singular metric induces the standard Euclidean
topology on the Julia set is flawed.

The goal of this paper is to fill in this gap in the literature in the
case of unicritical polynomials $f(z) = z^d + c$, combining the
results of Ma\~n\'e, Carleson, Jones, and Yoccoz with basic distortion
theorems and estimates of the hyperbolic and orbifold metric of
punctured disks. In particular, we are not combining two different
kinds of expansion like Carette did, which in the opinion of the
author makes this proof conceptually simpler. The proof that the
singular metric is topologically (in fact, H\"older) equivalent to the
Euclidean metric uses the fact that the basin of infinity is a John
domain, which was proved in \cite{CarlesonJonesYoccoz1994}.

Generalizing the proof to semihyperbolic polynomials with more than
one critical point seems to be non-trivial, for at least a couple of
different reasons: It is possible that the $\omega$-limit set of one
critical point contains another critical point, so the different parts
of the postcritical sets might not be ``separated'', and one might not
have expansion on the postcritical set with respect to the Euclidean
metric. Furthermore, the argument given in this paper for the fact
that the singular metric $\rho(z)|dz|$ induces the same topology on
the Julia set crucially uses that there are no bounded Fatou
components. Semihyperbolic polynomials with more than one critical
point can have bounded attracting domains, in which case our proof
does not work.

Whether expanding metrics exist for semihyperbolic rational functions,
or for polynomials satisfying weaker conditions such as Collet-Eckmann
are very interesting question which warrant further investigation. In
order to allow for metrics $\rho$ with singularities on all of $J(f)$,
the definition of expansion might have to be modified for these
cases.

Our main result is the following, proved in Section
\ref{sec:unicritical}.

\begin{restatable*}{theorem}{MainTheorem}
  \label{thm:main-thm}
  Assume that the polynomial $f(z) = z^d+c$ is semihyperbolic, but not
  hyperbolic. Then $f$ is expanding with respect to the metric
  $\rho(z)|dz|$, where $\rho(z) = 1+
  \frac{1}{\dist(z,P(f))^{1-1/d}}$. Explicitly, there exist $C>0$ and
  $\lambda > 1$ such that
  \[
    |(f^n)'(z)| \rho(f^n(z)) \ge C \lambda^n \rho(z)
  \]
  for all $z \in J(f) \setminus P(f)$ and all $n \ge 1$. Furthermore,
  the metric $\rho(z)|dz|$ is H\"older equivalent to the Euclidean
  metric on $\C$.
\end{restatable*}
\begin{remark}
  The neighborhood of the Julia set here is $U=\C$, so we consider the
  induced path metric $d_\rho$ on the whole plane. Note that the
  postcritical set $P(f)$ is a forward-invariant compact proper subset
  of the Julia set, which implies that it is nowhere dense in
  $J(f)$.

  In the case of subhyperbolic polynomials, the metric $\rho(z)|dz|$
  is a metric satisfying the conditions of
  Theorem~\ref{thm:subhyp-metric} with $\beta = 1-1/d$, and it is in
  fact bi-Lipschitz equivalent in a neighborhood of the Julia set to
  the canonical orbifold metric associated to $f$. For a definition of
  the canonical orbifold metric and its use in the proof of expansion
  in the subhyperbolic case, see \cite[Ch.~V]{CarlesonGamelin1993} or
  \cite[\S~19]{Milnor2006}. It seems that semihyperbolic polynomials
  which are not subhyperbolic do not have a useful associated orbifold
  structure, since $P(f) \cap J(f)$ is not a discrete set in that
  case.
\end{remark}

\section{Distortion estimates, hyperbolic and orbifold metrics}
\label{sec:distortion-estimates}
We will use the notation $B(z_0, r) = \{ z \in \C: |z-z_0| < r \}$,
$\bar{B}(z_0,r) = \overline{B(z_0,r)}$, and $S(z_0, r) = \partial
B(z_0, r) = \{ z \in \C: |z-z_0| = r \}$ for the open and closed disk,
and the circle of radius $r>0$ centered at $z_0 \in \C$, respectively.

Koebe's distortion theorems are usually stated for normalized
conformal maps, see e.g.\ \cite{Duren1983}. The following versions for
general conformal maps in disks are immediate consequences by
translation and rescaling.
\begin{theorem}[Koebe Distortion and 1/4-Theorem]
  Let $g:B(z_0, r) \to \C$ be univalent with $w_0 = g(z_0)$. Then
  \begin{equation}
    \frac{|g'(z_0)|}{\left(1+\frac{|z-z_0|}{r}\right)^2} \le \left|
      \frac{g(z)-g(z_0)}{z-z_0} \right| \le
    \frac{|g'(z_0)|}{\left(1-\frac{|z-z_0|}{r}\right)^2}  
  \end{equation}
  for all $z \in B(z_0, r) \setminus \{z_0\}$. Furthermore, the disk
  $B\left(w_0, \frac{|g'(z_0)| r}{4}\right)$ is contained in the image
  $g(B(z_0,r))$.
\end{theorem}
We will frequently apply this theorem to inverse branches of iterates
of $f$. We define the \emph{conformal radius} $r(U,z_0)$ of a simply
connected domain $U \subsetneq \C$ with respect to a point $z_0 \in U$
as the unique $r>0$ for which there is a conformal map
$g:B(0,r) \to U$ with $g(0)=z_0$ and $g'(0)=1$. If we denote the
hyperbolic metric of $U$ by $\delta_U(z) |dz|$, then conformal
invariance of the hyperbolic metric and the explicit formula
$\delta_{B(0,r)} = \frac{2r}{r^2-|z|^2}$ immediately give the
well-known relation
\begin{equation}
  \label{eq:conf-radius-hyp-metric}
  \delta_U(z_0) = \frac{2}{r(U,z_0)}
\end{equation}
between conformal radius and density of the hyperbolic
metric. Furthermore, by the Schwarz lemma we know that $U$ can not be
contained in any disk $B(z_0,r')$ with $r'<r$, so that
\begin{equation}
  \label{eq:conf-radius-diam}
  \diam U \ge r(U,z_0).
\end{equation}
It also immediately follows from the definition that the conformal
radius of the image of $U$ under a conformal map $g:U \to V$ with
respect to the point $w_0 = g(z_0)$ is
\begin{equation}
  \label{eq:conf-radius}
  r(V,w_0) = |g'(z_0)|r(U,z_0).
\end{equation}

An \emph{orbifold} is a topological space locally modeled on the
quotient of Euclidean space by a finite group. In the context of
complex analysis, we are mostly interested in \emph{Riemann orbifolds}
(as a generalization of Riemann surfaces) which are locally modeled on
the unit disk $\D$ or on the quotient of the unit disk by the group
$\Gamma_d$ of rotations of $d$-th roots of unity. For background on
Riemann orbifolds see \cite{McMullen1994}.

A point on the orbifold which corresponds to the fixed point $0$ of
$\Gamma_d$ in such a chart is called a \emph{cone point of order
  $d$}. If $\phi:U \to \D /\Gamma_d$ is such a chart, then
$\phi_d:U \to \D$, $\phi_d(z) = \phi(z)^d$ is a homeomorphism. If we
replace each chart $\phi$ by $\phi_d$, we get an atlas of a Riemann
surface, the \emph{underlying} Riemann surface $\cS$ of the orbifold
$\cO$. This also shows that every Riemann orbifold can be specified by
giving a Riemann surface $\cS$ and a discrete set of cone points $z_k$
with corresponding orders $d_k = d(z_k) \ge 2$. A map
$f:\cO \to \tilde\cO$ between orbifolds is an \emph{analytic orbifold
  map} if $f:\cS \to \tilde\cS$ is an analytic map between the
underlying Riemann surfaces satisfying
\begin{equation}
  d(f(z)) \mid d(z) \deg(f,z)
\end{equation}
for all $z \in \cO$, where $d(z)$ and $d(f(z))$ are the orders of the
cone points at $z$ and $f(z)$, respectively, with the convention that
$d(z) = 1$ for points which are not cone points. It is an
\emph{analytic orbifold covering map} if it is an analytic branched
covering from $\cS$ to $\tilde\cS$, with
\begin{equation}
  d(f(z)) = d(z) \deg(f,z)
\end{equation}
for all $z \in \cO$. It is known \cite[Theorem~A.2]{McMullen1994} that
the only Riemann orbifolds which do not have a Riemann surface cover
are the \emph{teardrop} ($\CC$ with one cone point of order $d \ge 2$)
and the \emph{spindle} ($\CC$ with two cone points of orders $d_1,d_2
\ge 2$ with $d_1 \ne d_2$). All other Riemann orbifolds have as
universal cover either the sphere, the plane, or the disk,
respectively called \emph{elliptic}, \emph{parabolic}, or
\emph{hyperbolic}. A Riemann orbifold $\cO$ with underlying Riemann
surface $\cS$ is hyperbolic iff its Euler characteristic
\[
  \chi(\cO) = \chi(\cS) - \sum_k \left( 1-\frac1{d(z_k)} \right)
\]
is negative. This shows that most orbifolds are hyperbolic, and that
in particular every orbifold whose underlying Riemann surface is
hyperbolic is itself hyperbolic. If $\cO$ is a hyperbolic Riemann
orbifold, and if $\phi: \D \to \cO$ is an analytic orbifold covering
map, then $\cO$ inherits a hyperbolic metric from pushing forward the
hyperbolic metric $\frac{2|dz|}{1-|z|^2}$ on the unit disk, and since
any two such analytic orbifold covering map differ by composition with
a hyperbolic isometry, this metric is independent of the choice of
$\phi$. We call this metric the \emph{hyperbolic orbifold metric} of
$\cO$ and denote it by $\delta_\cO(z) |dz|$, and the induced
\emph{hyperbolic orbifold distance} between points $z,w \in \cO$ by
$\delta_\cO(z,w)$.  As a special case, if $\cO = \cS$ is a Riemann
surface, the hyperbolic orbifold metric $\delta_\cS(z) |dz|$ is just
the hyperbolic metric of $\cS$. For convenience, we will also
introduce the \emph{pseudo-hyperbolic orbifold distance} $p_\cO(z,w)$,
defined in the unit disk by
\[
  p_\D(z,w) = \left| \frac{z-w}{1-\overline{w}z} \right|,
\]
and pushed forward by the analytic orbifold covering $\phi:\D \to \cO$
as usual. The pseudo-hyperbolic distance is a metric, but it is not
generated by a path metric and it is not complete if $\cO$ is not
compact. The hyperbolic and pseudo-hyperbolic distances are related by
\[
  \delta_\cO(z,w) = 2\tanh^{-1} p_\cO(z,w) = \log
  \frac{1+p_\cO(z,w)}{1-p_\cO(z,w)}.
\]
The next rather elementary, but very useful lemma compares the
hyperbolic metric and the hyperbolic orbifold metric in the case of
simply connected domains in the plane with a single cone point.
\begin{theorem}
  \label{thm:metric-comparison}
  Let $U \subsetneq \C$ be a simply connected domain, let $z_0 \in
  U$, $d \ge 2$, and let $\cO$ be the Riemann orbifold with
  underlying domain $U$ and a single cone point of order $d$ at
  $z_0$, Then
  \[
    \frac{\delta_\cO(z)}{\delta_{U}(z)} = F_d(p_U(z_0,z)) \qquad
    \text{for } z \in U \setminus\{z_0\},
  \]
  with
  \[
    F_d(t) = \frac{1-t^2}{dt^{1-1/d} (1-t^{2/d})} = \frac{t^{-1}
      -t}{d(t^{-1/d} - t^{1/d})} \qquad \text{for } 0 < t < 1.
  \]
  Furthermore, we have the estimates
  \[
    \frac{1}{d} \le  F_d(t) t^{1-1/d} \le 1 \qquad \text{for } 0 < t <
    1.
  \]
\end{theorem}
\begin{proof}
  Let $\psi: U \to \D$ be a conformal map from the underlying domain
  $U$ to $\D$ with $\psi(z_0) = 0$. Then $\psi$ is also an orbifold
  covering map of the orbifold $\cO$ to the orbifold $\cO_d$ with
  underlying domain $\D$ and a cone point of order $d$ at
  $0$. Analytic orbifold covering maps lift to hyperbolic isometries,
  so they are local isometries with respect to the respective
  hyperbolic orbifold metrics, i.e.,
  $\delta_{\cO_d}(\psi(z))|\psi'(z)| = \delta_{\cO}(z)$ for all $z \in
  \cO$. The same is true for the hyperbolic metrics of the underlying
  domains, i.e., $\delta_{\D}(\psi(z)) |\psi'(z)| = \delta_U(z)$ for
  all $z \in U$. This implies that
  \[
    \frac{\delta_\cO(z)}{\delta_U(z)} =
    \frac{\delta_{\cO_d}(\psi(z))}{\delta_{\D}(\psi(z))} 
  \]
  for all $z \in U$. Since conformal maps of simply connected domains
  are isometries with respect to the hyperbolic and pseudo-hyperbolic
  metrics, we also get that $p_U(z_0,z) = p_\D(0,\psi(z))$, so it is
  enough to prove the theorem for the case of $U = \D$, $z_0=0$, and
  $\cO = \cO_d$. In this case a hyperbolic orbifold covering
  $\phi:\D \to \cO$ is given by $\phi(z) = z^d$, so the hyperbolic
  orbifold metric of $\cO_d$ is given by
  \[
    \delta_{\cO_d}(z) = \frac{2}{d|z|^{1-1/d}(1-|z|^{2/d})}.
  \]
  This shows that 
  \[
    \frac{\delta_{\cO_d}(z)}{\delta_\D(z)} =
    \frac{1-|z|^2}{d|z|^{1-1/d} (1-|z|^{2/d})} =
    \frac{|z|^{-1}-|z|}{d(|z|^{-1/d} - |z|^{1/d})} = F_d(|z|)
  \]
  for all $z \in \D$, and since $|z| = p_\D(0,z)$ the first claim in
  the theorem follows. The claimed estimates follow easily from the
  explicit calculation
  \[
    F_d(t) t^{1-1/d} = \frac{1-t^2}{d(1-t^{2/d})} =  \frac{1 +
    t^{2/d} + t^{4/d} + \ldots + t^{(2d-2)/d}}{d}.
  \]
\end{proof}

\section{John domains, Gehring trees and singular metrics}
Roughly speaking, a John domain is a domain in which one can connect
any two points without getting too close to the boundary. This concept
was introduced in the context of elasticity by John in
\cite{John1961}, and the term ``John domain'' was coined by Martio and
Sarvas in \cite{MartioSarvas1979}. By now, John domains have found
many applications in various branches of geometric analysis, and there
are multiple different equivalent definitions. For an excellent
in-depth survey and the comparison of various definitions for both
bounded and unbounded domains see \cite{NakkiVaisala1991}.
There are some subtleties in the case of unbounded domains, where the
various different common definitions of John domains are not always
equivalent. However, the difficulties only appear in the case where
$\infty$ is a boundary point (when viewing these domains as subsets of
the Riemann sphere), whereas in the setting of our paper we always
have that $G \subset \CC$ is the basin of infinity for a polynomial,
so $\infty$ is an interior point of $G$.

\begin{definition}
  A domain $G \subsetneq \CC$ with $\infty \notin \partial G$ is a
  \emph{John domain} if there exists a constant $c \in (0,1)$ such
  that the following holds: For every $z_0,z_1 \in G$ there exists a
  simple path $\gamma\subset G$ from $z_0$ to $z_1$ in $G$ such that
  \[
    \dist(z,\partial G) \ge c \min \{ l(\gamma(z_0,z)),
    l(\gamma(z,z_1)) \}
  \]
  for all $z \in \gamma$, where $\gamma(z_0,z)$ and $\gamma(z,z_1)$
  are the sub-paths from $z_0$ to $z$ and from $z$ to $z_1$,
  respectively, and where $l(.)$ denotes Euclidean arc
  length. Any such path is called a \emph{John path} (or \emph{John
    arc} in the case where $\gamma$ is an arc) from $z_0$ to $z_1$.
\end{definition}

Note that we allow $z_0$ or $z_1$ to be infinite, in which case the
minimum on the right is the length of the bounded sub-path.

Since we are only dealing with unicritical semihyperbolic
polynomials, there are no bounded Fatou components, so the Julia set
$J = \partial G = \CC \setminus G$ is a continuum with empty interior
which does not disconnect the plane. Following \cite{LinRohde2018}, we
call such a continuum $J \subset \C$ whose complement $G = \CC
\setminus J$ is a John domain, a \emph{Gehring tree}. By the Riemann
mapping theorem there exists a conformal map $\phi: \Delta \to G$,
where $\Delta = \CC \setminus \overline{\D}$. For convenience we
choose the unique such map for which $\phi(\infty) = \infty$ and
$\lim\limits_{z \to \infty} \frac{\phi(z)}{z} \in (0,\infty)$. Since
boundaries of John domains are locally connected, $\phi$ extends
continuously to the closure $\overline{\Delta}$.  Following the
conventions in complex dynamics, we will call the parametrized curve
$\gamma(r) = \phi(r e^{i\theta})$, $1 \le r \le \infty$, the
\emph{external ray} of angle $\theta$. Note that external rays are
hyperbolic geodesics.

By \cite[Theorem~5.2]{NakkiVaisala1991}, see also
\cite[Theorem~4.1]{GehringHagMartio1989}, in the case of simply
connected hyperbolic domains, hyperbolic geodesics are John arcs with
some constant $c' \in (0,1)$ which only depends on the John constant
$c$. It is clear by a simple limit argument that hyperbolic geodesics
terminating at a point in $\partial G$ still satisfy the John
condition. In the rest of this section we adopt the convention that
$c \in (0,1)$ is a John constant which in particular works for all
hyperbolic geodesics, our ``canonical'' John arcs.

\begin{theorem}
  \label{thm:john-metric}
  Let $J \subset \C$ be a Gehring tree, $\alpha \in (0,1)$ be a
  constant, and let $\rho(z) = 1+\frac{1}{\dist(z,J)^\alpha}$. Then
  the path metric in $\C$ defined by $\rho(z)|dz|$ is H\"older
  equivalent to the Euclidean metric
\end{theorem}
In the context of semihyperbolic polynomials, the metric is defined
slightly differently, where the Julia set $J$ in the definition of the
metric is replaced by the postcritical set $P$. However, the same
result follows easily.
\begin{corollary}
  \label{cor:john-metric}
  Let $J \subset \C$ be a Gehring tree, $P \subset J$ be a compact
  non-empty set, $\alpha \in (0,1)$ be a constant, and let $\rho(z) =
  1+\frac{1}{\dist(z,P)^\alpha}$. Then the path metric in $\C$ defined
  by $\rho(z)|dz|$ is H\"older equivalent to the Euclidean metric.
\end{corollary}
\begin{proof}[Proof of Corollary~\ref{cor:john-metric}]
  This is a simple ``sandwiching'' argument.  Let us denote the two
  metrics by $\rho_P$ and $\rho_J$. We have that $|dz| \le \rho_P(z)
  |dz| \le \rho_J(z) |dz|$, and since $\rho_J(z) |dz|$ is H\"older
  equivalent to the Euclidean metric by Theorem~\ref{thm:john-metric},
  the same is true for $\rho_P(z)|dz|$.
\end{proof}
\begin{proof}[Proof of Theorem~\ref{thm:john-metric}]
  We will actually show that
  \begin{equation}
    \label{eq:hoelder-eq}
    |z_0-z_1| \le d_\rho(z_0,z_1) \le C |z_0-z_1|^{1-\alpha}
  \end{equation}
  for all $z_0, z_1 \in \C$ with $|z_0-z_1| < 1$. The lower estimate
  follows directly from $\rho(z) \ge 1$, so let us prove the upper
  estimate. For convenience, we will write $\delta(z) =
  \dist(z,\partial G)$. In this proof, $C_1, C_2, \ldots$ will denote
  positive constants which depend only on $\alpha$ and the John
  constant $c$. The sketch in Figure~\ref{fig:john-arcs} should help
  illuminate the argument.

  Let us first compare Euclidean and $\rho$-length of John
  arcs. Assume that $\gamma:[0,L] \to \overline{G}$ is an arc length
  parametrization of a hyperbolic geodesic between $z_0 = \gamma(0)$
  and $z_1 = \gamma(L)$, where we allow $z_0$ or $z_1$ to be on the
  boundary of $G$. Let $\gamma_0 = \gamma|_{[0,L/2]}$ be the first
  half of this geodesic segment and $\gamma_1 = \gamma|_{[L/2,L]}$ the
  second half. Then
  \begin{equation*}
    l_\rho(\gamma_0)-\frac{L}{2}
     = \int_{\gamma_0} \frac{|dz|}{\delta(z)^\alpha}
    = \int_0^{L/2} \frac{dt}{\delta(\gamma(t))}
    \le \int_0^{L/2} \frac{dt}{(ct)^\alpha}
    = \frac{L^{1-\alpha}}{(1-\alpha) 2^\alpha c^\alpha},
  \end{equation*}
  and the same estimate holds for $l_\rho(\gamma_1)$, so that
  \begin{equation}
    \label{eq:hoelder-length}
    l_\rho(\gamma) \le l(\gamma) + C_1 l(\gamma)^{1-\alpha}.
  \end{equation}
  This estimate is already sufficient to show that $d_\rho$ is
  H\"older equivalent to the internal Euclidean metric of
  $G$. However, two points $z_0,z_1 \in G$ with $|z_0-z_1|$ small can
  have a large internal Euclidean distance. (As an example, if
  $J=[-1,1]$, and $z_0=\epsilon i$, $z_1 = -\epsilon i$, we have
  $|z_0-z_1|=2\epsilon$, whereas the internal distance in $G = \C
  \setminus [-1,1]$ is always $\ge 2$.)
  
  \begin{figure}[t]
    \centering
    \includegraphics[width=.8\textwidth]{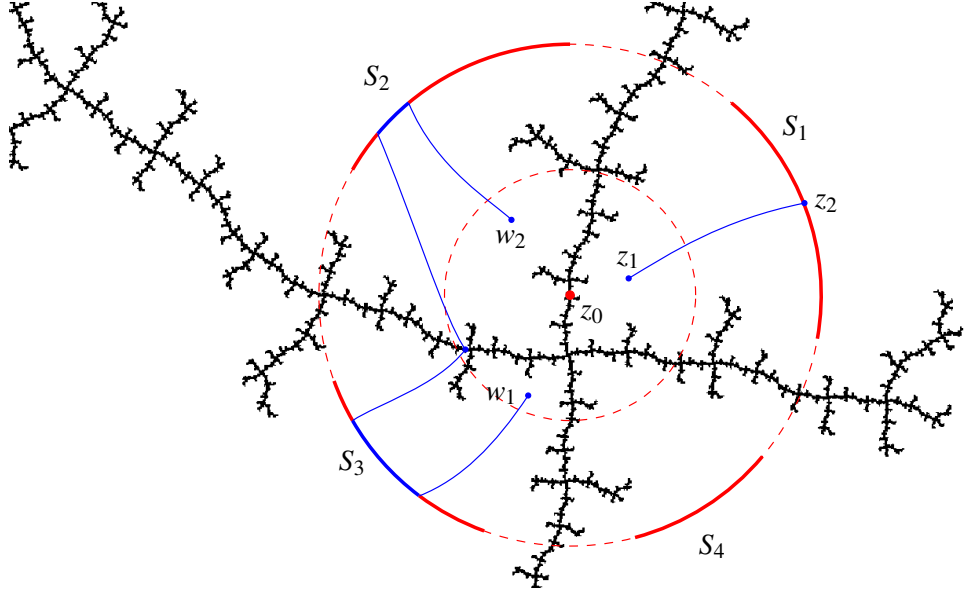}
    \caption[Sketch of construction]{Sketch of the construction in the
      proof for an example of a Gehring tree, part of the
      subhyperbolic Julia set of $f(z)=z^2+c$ with $c \approx
      0.419643+0.606291i$ (produced with Curt McMullen's software from
      \protect\url{http://math.harvard.edu/~ctm/programs/index.html}).
      Dashed circles are $S(z_0,r)$ and $S(z_0,2r)$, solid circular
      arcs are $S_k \subseteq S(z_0,2r)$. Part of an external ray is
      drawn from $z_1$ to $z_2$. All external rays from
      $\bar{B}(z_0,r)$ exit $B(z_0,2r)$ through one of the arcs
      $S_k$. The number of arcs, the $\rho$-length of the arcs and the
      parts of external rays contained in $|z-z_0| \le 2r$ are all
      controlled by the John constant and $r$. The path between $w_1$
      and $w_2$ illustrates how any two points in the $r$-neighborhood
      can be connected in $\bar{B}(z_0,2r)$ by a controlled number of
      external rays and subarcs of $S_k$, each of controlled length,
      connecting through points in $B(z_0,r) \cap J$.}
    \label{fig:john-arcs}
  \end{figure}
  
  As a next step we will find an upper bound for the $\rho$-diameter
  of Euclidean disks centered at points in $J$.  Fix $z_0 \in J$ and
  $r \in (0,1)$. Consider an arbitrary $z_1 \in \bar{B}(z_0,r)$
  and the external ray $\gamma$ from $z_1$ up to its first
  intersection with $S(z_0,2r)$ at some point $z_2$. (In the case
  where $z_1 \in \partial G$ is the landing point of more than one
  ray, let $\gamma$ be one of these rays.) Then $l(\gamma) \ge
  |z_2-z_1| \ge r$, so by the John property $B(z_2,cr) \subseteq
  G$. On the other hand, we know from the John property that
  $l(\gamma) \le \frac{2r}{c}$, since $\delta(z_2) \le 2r$, so by
  \eqref{eq:hoelder-length} we have that
  \begin{equation}
    \label{eq:ext-ray-length}
    l_\rho(\gamma) \le \frac{2r}{c}+ C_1
    \left(\frac{2r}{c}\right)^{1-\alpha} \le \frac{2r^{1-\alpha}}{c}+ C_1
    \left(\frac{2r}{c}\right)^{1-\alpha} = C_2 r^{1-\alpha}.
  \end{equation}
  The set
  \begin{equation*}
    K = \{ z \in S(z_0, 2r): \delta(z) \ge cr \}    
  \end{equation*}
  is compact, and so it can be covered by finitely many disks of
  radius $cr/2$ centered at points in $K$. The union of these disks
  has finitely many connected components $U_1, \ldots, U_N$, and their
  intersection with the circle $S(z_0, 2r)$ consists of finitely many
  disjoint circular arcs $S_1, \ldots, S_N$, each with an arc length
  of at least $cr$, so that $N \le 4\pi/c$, an upper bound depending 
  only on the John constant $c$. We also have that
  \begin{equation}
    \label{eq:circle-length}
    l_\rho(S_j) = l(S_j) + \int_{S_j} \frac{|dz|}{\delta(z)^\alpha} \le
    \left[1+\left(\frac{2}{cr}\right)^\alpha\right] l(S_j) \le
    \left[\frac{1}{r^\alpha}+\frac{2^\alpha}{c^\alpha
        r^{\alpha}}\right] 4\pi r = C_3 r^{1-\alpha}.
  \end{equation}
  Now define $B_j$ to be the set of points $z_1 \in \bar{B}(z_0,r)$
  for which an external ray\footnote{For $z_1 \in G$ there is a unique
    external ray, but we allow $z_1 \in \partial G$, in which case
    there might be several.} from $z_1$ intersects
  $\overline{U_j}$. Since every such ray intersects $K$, and $K
  \subset \bigcup_{j=1}^N U_j$, we have that $\bar{B}(z_0,r) =
  \bigcup_{j=1}^N B_j$. Any two points $z_1, z_2 \in B_j$ can be
  joined by a path $\gamma$ in $\bar{B}(z_0,2r)$ by joining the
  external rays from $z_1$ and $z_2$ to $S_j$ with an arc of $S_j$, of
  total $\rho$-length $l_\rho(\gamma) \le C_4r^{1-\alpha}$, by
  \eqref{eq:ext-ray-length} and \eqref{eq:circle-length}, so that
  $\diam_\rho B_j \le C_4 r^{1-\alpha}$. Since $\bar{B}(z_0,r) =
  \bigcup_{j=1}^N B_j$ is a connected union we have that $\diam_\rho
  \bar{B}(z_0,r) \le C_5 r^{1-\alpha}$ with $C_5 = NC_4$.

  Now fix any $z_0,z_1 \in \C$ with $r = |z_0-z_1| \in (0,1)$. Then one
  of the following is true.
  \begin{casesp}
  \item There exists $z_2 \in \bar{B}(z_0,2r) \cap J$.
    
    In this case, $\bar{B}(z_0,r) \subset \bar{B}(z_2, 3r)$, so
    $\diam_\rho B(z_0,r) \le \diam_\rho B(z_2,3r) \le C_6
    r^{1-\alpha}$.
  \item $\bar{B}(z_0,2r) \cap J = \emptyset$.

    In this case $\delta(z) \ge r$ for every $z \in \bar{B}(z_0,r)$,
    so if $\gamma \subset \bar{B}(z_0,r)$ is a straight line segment,
    then $l_\rho(\gamma) \le l(\gamma)(1+ r^{-\alpha}) \le
    4r^{1-\alpha}$, so $\diam_\rho B(z_0,r) \le 4 r^{1-\alpha}$.
  \end{casesp}
  These cases combined show in particular that $d_\rho(z_0,z_1) \le C
  |z_0-z_1|^\alpha$ with $C = \max \{C_6,4\}$, proving
  \eqref{eq:hoelder-eq} and finishing the proof.
\end{proof}

\section{Semihyperbolicity}
\label{sec:semihyp}
Recall from the introduction that a polynomial $f$ is
\emph{semihyperbolic} if each critical point in the Fatou set of $f$
is attracted by an attracting periodic cycle, and each critical point
in the Julia set of $f$ is non-recurrent. Following ideas of Ma\~n\'e
from \cite{Mane1993}, Carleson, Jones, and Yoccoz in
\cite{CarlesonJonesYoccoz1994} showed the following.
\begin{theorem}[Carleson, Jones, Yoccoz]
  \label{thm:CJY}
  Let $f$ be a polynomial of degree $\ge 2$. Then the following are
  equivalent:
  \begin{enumerate}
  \item $f$ is semihyperbolic.
  \item There exists $\epsilon > 0$, $C>0$, $\theta \in (0,1)$, and $D
    \in \N$ such that for all $x \in J(f)$, for all $n \in \N$, and for
    all components $U$ of $f^{-n}(B(x,\epsilon))$, we have that $\diam U
    \le C \theta^n$, and that $\deg(f^n, U) \le D$.
  \item $A_\infty(f)$ is a John domain.
  \item $A_\infty(f)$ is a John domain and every bounded Fatou domain is
    a quasidisk.
  \end{enumerate}
\end{theorem}

Let us say that a
\emph{backward orbit} of a domain $U \subset \CC$ is a sequence of
domains $(U_n)$ such that $U_0 = U$, and $U_{n}$ is a connected
component of $f^{-1}(U_{n-1})$ for $n \ge 1$. In particular this
implies that the restriction $f|_{U_n}:U_{n} \to U_{n-1}$ is a
branched covering for every $n \ge 1$. The sequence $(d_n)$ defined by
$d_n = \deg f|_{U_n}$ for $n \ge 1$ is the \emph{sequence of degrees}
of this backward orbit. Notice that $f^n|_{U_n}:U_n \to U$ is a
branched covering of degree $\prod_{k=1}^n d_k$. We say that a
backward orbit $(U_n)$ is \emph{univalent} if $d_n=1$ for all $n$.

The following theorem is a slightly refined version of Theorem~2.1 in
\cite{CarlesonJonesYoccoz1994}.
\begin{theorem}
  \label{thm:CJY-refined}
  Let $f$ be a polynomial of degree $d \ge 2$ without parabolic
  periodic points and without recurrent critical points. Let $c_1,
  \ldots c_r$ be the distinct critical points of $f$, of
  multiplicities $m_1, \ldots, m_r$, and let $\epsilon_0 > 0$ be
  arbitrary. Then there exists $\epsilon > 0$, $C_0>0$, and $\theta \in
  (0,1)$ such that the following holds: For all $z \in J(f)$, and for
  any backward orbit $(U_n)$ of $U = B(z,\epsilon)$ we have
  \begin{enumerate}
  \item $\diam U_n \le \min\{C_0 \theta^n,\epsilon_0\}$.
  \item Each $U_n$ contains at most one critical point, and each
    critical point $c_j$ is contained in at most one $U_n$.
  \item Each $U_n$ is simply connected, and the sequence of degrees
    $(d_n)$ satisfies $d_n = 1$ if $U_n$ does not contain a critical
    point, and $d_n = m_j + 1$ if $c_j \in U_n$.
  \item The degree of $f^n|_{U_n}:U_n \to U$ is bounded by $D =
    \prod_{j=1}^r (m_j+1) \le 2^{d-1}$.
  \end{enumerate}
\end{theorem}
\begin{proof}
  Essentially all these properties directly follow from the proof in
  \cite{CarlesonJonesYoccoz1994}, but we will show how all of these
  follow from their stated results as well as the results from
  \cite{Mane1993}. Property (1) combines results from these two
  papers with a compactness argument.  Properties (3) and (4) are
  immediate consequences of (2), which is the main point of this
  refinement. We may assume by possibly choosing $\epsilon_0$ smaller
  that for every critical point $c \in J(f)$ the
  $2\epsilon_0$-neighborhood $B(c,2\epsilon_0)$ does not contain any
  other critical point, and no point of the forward orbit of
  $c$. Since $f$ has no parabolic periodic points and no recurrent
  critical points, by \cite[Theorem~II~(a)]{Mane1993}, for every $z
  \in J(f)$ there exists a neighborhood $U_z$ of $z$, so that every
  connected component of $f^{-n}(U_z)$ has diameter $ <
  \epsilon_0$. Then $\{U_z\}_{z \in J(f)}$ is an open cover of the
  compact Julia set of $f$, so it has a finite subcover
  $\{U_{z_k}\}_{k=1}^m$, with some Lebesgue number $\epsilon > 0$,
  i.e., such that for any $z_0 \in J(f)$, the $\epsilon$-neighborhood
  $B(z_0,\epsilon)$ is contained in $U_{z_k}$ for some $1 \le k \le
  m$.

  Now let $z \in J(f)$ be arbitrary, and let $(U_n)$ be a backward
  orbit of $U = B(z,\epsilon)$. Then $\diam U_n < \epsilon_0$ for all
  $n \ge 0$, and so by our assumption on $\epsilon_0$ each $U_n$ can
  contain at most one critical point. If some critical point $c$ is
  contained in $U_n$ and $U_{n+j}$ with $n,j \in \N$, $j \ge 1$, then
  $U_{n+j}$ contains both $c$ and $f^j(c)$, contradicting our
  assumption on $\epsilon_0$. Combined this shows (2).

  The fact that there exist constants $C_0>0$ and $\theta \in (0,1)$
  (independent of $z$) with $\diam U_n \le C_0 \theta^n$ for all
  $n \ge 0$ follows from
  \cite[Theorem~2.1]{CarlesonJonesYoccoz1994}, where we might have to
  allow for making $\epsilon$ smaller.

  Preimages of simply connected domains under polynomials are simply
  connected, so all $U_n$ are simply connected. The degree of
  compositions of branched covers is the product of the degrees, and
  since each critical point appears at most once in the sequence
  $(U_n)$, property (4) immediately follows.
\end{proof}

\section{Unicritical polynomials}
\label{sec:unicritical}
Up to conjugation, unicritical polynomials are polynomials of the form
$f(z) = z^d + c$ with $d \ge 2$, and critical point $0$ of
multiplicity $d-1$. Such a polynomial is semihyperbolic iff
$0 \notin \omega(0)$.  The main technical result in this section is
Theorem~\ref{thm:expansion}, which shows that a semihyperbolic
polynomial $f(z)=z^d+c$ is expanding with respect to the singular
metric $\frac{|dz|}{\dist(z,P(f))^{1-1/d}}$. Here is a rough sketch of
the proof strategy: In the unicritical case, the statement of
Theorem~\ref{thm:CJY-refined} takes a particularly simple form. Every
backward orbit $(U_n)$ of a small disk $U$ is either univalent, or
there is exactly one $n_0$ for which $d_{n_0} = d$, and $d_n=1$ for
all $n \ne n_0$.  In the case where $f^n$ does not have a critical
point in $U_n$, it is a conformal map from $U_n$ to $U$, and thus a
hyperbolic isometry. If $f^n$ has a critical point in $U_n$, then
there is only one, and $f_n$ is a branched cover from $U_n$ to $U$ of
degree $d$, with exactly one branch point $w_n$ of local degree
$d$. Equivalently, $f^n$ is an orbifold covering map from $U_n$ to
$\cO$, where $\cO$ is the orbifold with underlying domain $U$ and a
single cone point of order $d$ at the critical value $w_0 =
f^n(w_n)$. Combining the exponential shrinking of the diameters of
$U_n$ from Theorem~\ref{thm:CJY-refined} with Koebe distortion
estimates and the comparison of orbifold and hyperbolic metrics from
Section~\ref{sec:distortion-estimates}, we obtain the desired metric
distortion estimates for $f$ with respect to the metric $\rho(z)|dz|$.
\begin{theorem}
  \label{thm:expansion}
  Assume that the polynomial $f(z) = z^d+c$ is semihyperbolic but not
  hyperbolic. Then $f$ is expanding with respect to the metric
  $\sigma(z)|dz|$, where
  $\sigma(z) = \frac{1}{\dist(z,P(f))^{1-1/d}}$. Explicitly, there
  exist $C>0$ and $\lambda > 1$ such that
  \begin{equation}
    \label{eq:sigma-expansion}
    |(f^n)'(z)| \sigma(f^n(z)) \ge C \lambda^n \sigma(z)
  \end{equation}
  for all $z \in J(f) \setminus P(f)$ and all $n \ge 1$.
\end{theorem}

\begin{proof}
  Let us first note that $P(f)$ is a forward-invariant compact proper
  subset of the Julia set $J(f)$, so it is in particular nowhere dense
  in $J(f)$. By establishing \eqref{eq:sigma-expansion} for all
  $z \in J(f) \setminus P(f)$, we thereby show that it holds on a
  dense subset of the Julia set, which establishes that $f$ is
  expanding with respect to $\sigma$.

  Fix constants $\epsilon>0$, $C_0>0$, and $\theta \in (0,1)$ from
  Theorem~\ref{thm:CJY-refined}. We define
  $\lambda := \theta^{-1/d} > 1$. In the course of this proof,
  $C_1, C_2, \ldots$ will denote positive constants which depend only
  on $\epsilon$, $C_0$, $\theta$, and $d$.

  Pick an arbitrary $z_0 \in J(f)$, and a backward orbit
  $(z_n)_{n=0}^\infty$ of $z_0$, i.e., a sequence satisfying
  $f(z_n) = z_{n-1}$. In order to prove the theorem, we will show that
  the inequality \eqref{eq:sigma-expansion} holds for all $z_n$, as
  long as $z_n \notin P(f)$, with constants only depending on
  $\epsilon$, $C_0$, $\theta$, and $d$. Let
  $U_0 = U = B(z_0, \epsilon)$, and let $U_n$ be the connected
  component of $f^{-n}(U)$ containing $z_n$.  Let $r_n = r(U_n,z_n)$
  be the conformal radius of $U_n$ with respect to $z_n$. By
  definition, Theorem~\ref{thm:CJY-refined}, and inequality
  \eqref{eq:conf-radius-diam} we have that
  \begin{equation}
    \label{eq:ExpShrink}
    r_0 = \epsilon \qquad \text{and} \qquad r_n \le \diam
    U_n \le  C_0 \theta^n
  \end{equation}
  for all $n \ge 1$. Now we fix some arbitrary $n$, assume that
  $z_n \notin P(f)$, and consider the following three cases separately
  \begin{enumerate}
  \item $f^n$ maps $U_n$ conformally onto $U$, and $P(f) \cap U_n =
    \emptyset$.
  \item $f^n$ maps $U_n$ conformally onto $U$, and $P(f) \cap U_n \ne 
    \emptyset$.
  \item $f^n$ has a critical point in $U_n$.
  \end{enumerate}
  In cases (1) and (2) we know from \eqref{eq:conf-radius} that
  \begin{equation}
    \label{eq:f-n-prime}
    |(f^n)'(z_n)| = \frac{r_0}{r_n}
  \end{equation}
  \begin{casesp}
  \item $f^n$ maps $U_n$ conformally onto $U$, and $P(f) \cap U_n =
    \emptyset$.

    In this case Koebe's 1/4-theorem shows that
    $\dist(z_n, P(f)) \ge \frac{r_n}{4}$, so that
    $\sigma(z_n) \le \left(\frac{4}{r_n}\right)^{1-1/d}$. We also
    observe that $m = \min\limits_{z \in J(f)} \sigma(z) > 0$, so
    $\sigma(z_0) \ge m$. Combining these estimates with
    \eqref{eq:f-n-prime}, we get that
    \begin{equation}
      \begin{split}
        \frac{|(f^n)'(z_n)| \sigma(z_0)}{\sigma(z_n)} & \ge m
        \frac{r_0}{r_n} \left(\frac{r_n}{4}\right)^{1-1/d} = \frac{m
          \epsilon}{2 \cdot
          4^{1-1/d}} r_n^{-1/d} \\
        & \ge \frac{m \epsilon}{2 \cdot 4^{1-1/d} C_0^{1/d}}
        \left(\theta^{-1/d}\right)^n = C_1 \lambda^n
      \end{split}
    \end{equation}
  \item $f^n$ maps $U_n$ conformally onto $U$, and $P(f) \cap U_n \ne
    \emptyset$.
    
    We are assuming $z_n \notin P(f)$, so there exists
    $w_n \in P(f) \cap U_n$ with $\dist(z_n, P(f)) = |z_n-w_n| > 0$,
    so that $\sigma(z_n) = \frac{1}{|z_n-w_n|^{1-1/d}}$. Then
    $w_0 = f^n(w_n) \in P(f)$, so
    $\sigma(z_0) \ge \frac{1}{|z_0 - w_0|^{1-1/d}}$, and
    \begin{equation}
      \frac{|(f^n)'(z_n)| \sigma(z_0)}{\sigma(z_n)} \ge
      \frac{r_0}{r_n} \left|\frac{z_n - w_n}{z_0-w_0}\right|^{1-1/d}.
    \end{equation}
    By Koebe's distortion theorem applied to the branch of $f^{-n}$
    mapping $U_0$ to $U_n$, and using \eqref{eq:f-n-prime}, we get that
    \begin{equation}
      \left|\frac{z_n - w_n}{z_0-w_0}\right| \ge
      \frac{r_n}{r_0\left(1+\frac{|z_0-w_0|}{r_0}\right)^2} \ge
      \frac{r_n}{4r_0}.
    \end{equation}
    Combining these inequalities, we have that
    \begin{equation}
      \begin{split}
        \frac{|(f^n)'(z_n)| \sigma(z_0)}{\sigma(z_n)} & \ge
        \frac{r_0}{r_n}\left(\frac{r_n}{4r_0}\right)^{1-1/d} =
        \frac{1}{4^{1-1/d}} \left(\frac{r_0}{r_n}\right)^{1/d} \\
        & \ge \frac{1}{4^{1-1/d}} \left( \frac{\epsilon}{2C_0 \theta^n}
        \right)^{1/d}
        \ge \frac{\epsilon^{1/d}}{4^{1-1/d} 2^{1/d} C_0^{1/d}}
        \left(\theta^{-1/d}\right)^n \\
        &
        = C_2 \lambda^n.
      \end{split}
    \end{equation}
  \item $f^n$ has a critical point in $U_n$.

    By Theorem~\ref{thm:CJY-refined}, there exists
    $n_0 \in \{1,\ldots,n\}$ such that $0 \in U_{n_0}$, whereas
    $0 \notin U_j$ for $j \in \{1,\ldots,n\} \setminus\{n_0\}$.
    Furthermore, the theorem guarantees that no backward orbit of
    $U_{n_0}$ contains the critical point $0$, so
    $U_j \cap P(f) =\emptyset$ for $j \ge n_0$, and in particular
    $U_n \cap P(f) = \emptyset$. Applying Koebe's 1/4-theorem then
    gives $\dist(z_n, P(f)) \ge \frac{r_n}{4}$. We also know that
    $w_0 \in U_0 \cap P(f)$, so $\dist(z_0, P(f)) \le |z_0-w_0|$. In
    terms of the metric density $\sigma$, these inequalities are
    \begin{equation}
      \label{eq:case3-ineq2}
      \sigma(z_n) \le \left(\frac{4}{r_n}\right)^{1-1/d} \quad
      \text{and} \quad
      \sigma(z_0) \ge \frac{1}{|z_0-w_0|^{1-1/d}}.
    \end{equation}

    As a further consequence, $f^n|_{U_n}$ is a composition of $n-1$
    conformal maps $f|_{U_j}:U_j \to U_{j-1}$ for $j \ne n_0$, and one
    branched covering $f|_{U_{n_0}}:U_{n_0} \to U_{n_0-1}$ with
    exactly one branch point at $z=0$. Let $w_n \in U_n$ be the unique
    point such that $f^{n-k}(w_n) = 0$, and let $w_0 = f^k(0)$. If we
    endow $U_0$ with an orbifold structure by placing a single cone
    point of index $d$ at $w_0$, denoting the resulting orbifold by
    $\cO_0$, then $f^n|_{U_n}:U_n \to \cO_0$ is an orbifold covering
    map, and thus it is a (local) isometry with respect to the
    corresponding orbifold metrics $\delta_{U_n}$ (which is just the
    hyperbolic metric of $U_n$) and $\delta_{\cO_0}$. Explicitly,
    \begin{equation}
      \label{eq:orbifold-isometry1}
      \delta_{U_n}(z) = \delta_{\cO_0}(f^n(z)) |(f^n)'(z)|
    \end{equation}
    for all $z \in U_n \setminus\{ w_n \}$. Applying this for $z=z_n$
    and using the relation \eqref{eq:conf-radius-hyp-metric} between
    conformal radius and the density of the hyperbolic metric, we get
    \begin{equation}
      \label{eq:orbifold-isometry2}
      \frac{2}{r_n} = \delta_{U_n}(z_n) = \delta_{\cO_0}(z_0) |(f^n)'(z_n)|.
    \end{equation}
    Now let $t = p_{U_0}(z_0, w_0) = \frac{|z_0-w_0|}{r_0}$ denote the
    pseudo-hyperbolic distance between $z_0$ and $w_0$ in
    $U_0 = B(z_0,r_0)$. Applying Theorem~\ref{thm:metric-comparison}
    to equation \eqref{eq:orbifold-isometry2} yields
    \begin{equation}
      \begin{split}
        \frac{2}{r_n} & = \delta_{U_0}(z_0) F_d(t) |(f^n)'(z_n)|
        \le \frac{2|(f^n)'(z_n)|}{r_0 \cdot t^{1-1/d}}  =
         \frac{2 |(f^n)'(z_n)|}{r_0^{1/d} \cdot |z_0-w_0|^{1-1/d}},
      \end{split}
    \end{equation}
    so that
    \begin{equation}
      \label{eq:case3-ineq1}
      |(f^n)'(z_n)| \ge \frac{r_0^{1/d} \cdot |z_0-w_0|^{1-1/d}}{r_n}.
    \end{equation}
    
    Combining \eqref{eq:case3-ineq1} and
    \eqref{eq:case3-ineq2}, we get that
    \begin{equation}
      \begin{split}
        \frac{|(f^n)'(z_n)| \sigma(z_0)}{\sigma(z_n)} & \ge
        \frac{r_0^{1/d}}{r_n} \left(\frac{r_n}{4}\right)^{1-1/d} =
        \frac{r_0^{1/d}}{4^{1-1/d} \cdot r_n^{1/d}} \\
        & \ge \frac{\epsilon^{1/d}}{4^{1-1/d} \cdot C_0^{1/d}}
        \left(\theta^{-1/d}\right)^n = C_3 \lambda^n,
      \end{split}
    \end{equation}
  \end{casesp}
  finishing the proof that $f$ is expanding on $J(f)$ with respect to
  $\sigma(z)|dz|$.
\end{proof}

\MainTheorem
\begin{proof}
  By Theorem~\ref{thm:CJY}, we know that $A_f(\infty)$ is a John
  domain. By definition, semihyperbolic polynomials do not have
  parabolic periodic points, and from \cite{Mane1993} we know that
  they do not have Siegel disks either. This shows that any bounded
  Fatou component for semihyperbolic polynomials is in the basin of
  some attracting or superattracting periodic cycle, and any such
  basin contains a critical point. In our situation we have only one
  critical point, and by the assumption that $f$ is not hyperbolic, it
  is not contained in an attracting or superattracting basin, which
  shows that $f$ has no bounded Fatou components, so that the Julia
  set $J(f)$ is a Gehring tree.  By Corollary~\ref{cor:john-metric},
  the metric $\rho(z)|dz|$ is then H\"older equivalent to the
  Euclidean metric on $\C$. By Theorem~\ref{thm:expansion}, $f$ is
  expanding with respect to $\sigma(z)|dz|$, where
  $\sigma(z) = \frac{1}{\dist(z,P(f))^{1-1/d}}$. Since the Julia set
  $J(f)$ is compact, there exists some $R>0$ such that
  $J(f) \subseteq B(0,r)$.  For $|z| \le R$ we have that
  $\sigma(z) \le \rho(z) \le (2R+1)\sigma(z)$, so that
  \begin{equation*}
    |(f^n)'(z)| \rho(f^n(z)) \ge |(f^n)'(z)| \sigma(f^n(z)) \ge C
    \lambda^n \sigma(z) \ge C' \lambda^n \rho(z)
  \end{equation*}
  with $C'=\frac{C}{2R+1}$, which shows that $f$ is expanding with
  respect to $\rho(z)|dz|$ as well.
\end{proof}

\bibliography{expanding}
\bibliographystyle{amsalpha}

\end{document}